\documentclass[reqno,11pt]{amsart}

\usepackage{epsf}
\usepackage{graphics}
\usepackage{graphicx, color, mathtools}
\usepackage{amssymb}
\usepackage{amsmath}
\usepackage{mathcomp}
\usepackage{wasysym}

\date{}

\theoremstyle{plain}
\newtheorem{theorem}{Theorem}

\newtheorem{lemma}{Lemma}
\newtheorem{proposition}{Proposition}
\newtheorem{question}{Question}

\theoremstyle{definition}

\theoremstyle{remark}

\newtheorem{remark}{Remark}

\def\R{{\mathbb R}}

\title{Distortion of spheres and surfaces in space}

\author{Sebastian Baader, Luca Studer, Roger Z\"ust}

\begin{document}

\begin{abstract} It is known that the surface of a cone over the unit disc with large height has smaller distortion than the standard embedding of the 2-sphere in $\R^3$. In this note we show that distortion minimisers exist among convex embedded 2-spheres and have uniformly bounded eccentricity. Moreover, we prove that $\pi/2$ is a sharp lower bound on the distortion of embedded closed surfaces of positive genus. 
\end{abstract}

\maketitle

\section{Introduction}

The distortion of a path-connected subset $A \subset \R^n$ is the largest ratio between the intrinsic distance and the Euclidean distance of pairs of points in $A$:
$$\delta(A)=\sup_{p,q \in A} \frac{d_A(p,q)}{|p-q|} \, .$$
Here the intrinsic distance $d_A(p,q)$ of two points $p,q \in A$ is defined by minimising the length of paths connecting $p$ and $q$. 
In the special case of circles embedded in $\R^n$, there is a universal lower bound on the distortion, $\pi/2$, which is attained by round circles only. Surprisingly, this bound does not carry over to embeddings of higher-dimensional spheres in Euclidean space, as shown by the following two statements. The first one was observed by Gromov in \cite[Chapter~9]{G1}.

\begin{theorem} \quad
\begin{enumerate}
\item[(i)] The distortion of the surface of a cone over the unit disc with sufficiently large height (e.g.~$h=3$)
is strictly smaller than $\pi/2$. 
\item[(ii)] The distortion of the boundary of a regular $(n+1)$-dimensional simplex is $\sqrt{2+2/n}$. In particular, for $n \geq 5$, the boundary of the regular $(n+1)$-simplex has a lower distortion than the round sphere.
\end{enumerate}
\end{theorem}

The eccentricity of an embedded sphere is the ratio of its circumradius and its inradius. Embedded spheres of distortion $\leq \pi/2$ in $\R^3$ can have arbitrarily large eccentricity, as shown by rotationally symmetric ellipsoids and long cones. In contrast, distortion minimisers among convex spheres have uniformly bounded eccentricity.

\begin{theorem}
There exist convex embeddings of the 2-sphere in $\R^3$ minimising the distortion among all convex embeddings of the 2-sphere in $\R^3$. Moreover, these minimisers have uniformly bounded eccentricity.
\end{theorem}

The existence of minimisers for non-convex embeddings of the 2-sphere, as well as for surfaces of positive genus in $\R^3$, remains to be settled. Nevertheless, for the latter, we will determine the best possible lower bound on the distortion among all embeddings.

\begin{proposition}
The distortion of an embedded closed surface of genus $g \geq 1$ into $\R^3$ can be arbitrarily close to $\pi/2$, but not smaller. 
\end{proposition}

This leaves the 2-sphere as a challenging special case among surfaces. The proofs of Theorems~1,2 and Proposition~1 are contained in Sections~2,3 and~5, respectivly. Proposition~1 requires a lower bound on the distortion of subsets of $\R^n$ containing systoles, which we derive in Section~4. We conclude with three basic questions on the distortion of embedded spheres.

\begin{question}
What is the smallest possible distortion of an $n$-sphere embedded in $\R^{n+1}$?
\end{question}

\begin{question}
Does the smallest possible distortion of an $n$-sphere embedded in $\R^{n+k}$ depend on $k \geq 1$?
\end{question}

\begin{question}
Can the distortion of an $n$-sphere embedded in $\R^{n+k}$ be $\sqrt{2}$, or even smaller?
\end{question}

This work was triggered by Misev and Pichon's recent characterisation of superisolated  hypersurface singularity with finite distortion~\cite{MP}. We thank Filip Misev for explaining us the details of their result.

\section{Upper bounds}

In this section we compute the distortion of a cone in $\R^3$, and of the boundary of the standard $(n+1)$-simplex.
We consider the surface of a cone over the disc of radius $0<r<1$, $S(r)= C(r) \cup D(r)$, where 
\begin{align*}
C(r) &= \left\{(x,y,z)\in \R^3: x^2+y^2=\tfrac{r^2}{1-r^2}z^2, \ 0\leq z \leq \sqrt{1-r^2}\right\} \, , \\
D(r) &= \left\{(x,y,z)\in \R^3: x^2+y^2\leq r^2, \ z=\sqrt{1-r^2}\right\} \, ,
\end{align*}
see Figure~\ref{f1}. Note that $S(r)$ is a topologically embedded $2$-sphere.
\begin{figure}[h]
\def\svgwidth{200pt}
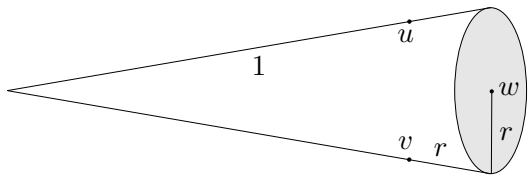
  \caption{The closed cone $S(r)$ with $r \approx 0.169\ldots$}
  \label{f1}
\end{figure}

The following proposition implies the first part of Theorem~1.

\begin{proposition}
\label{cone}
The distortion of $S(r)$ is smaller than $\pi/2$ if and only if $r<1-\pi^2/8$ and achieves its minimum $\delta(S(r_0))=1.552\ldots$ at the unique solution $r_0=0.169\ldots$ of 
\[
\frac{\sin(\pi r /2)}{r}=\frac{\sqrt{2}}{\sqrt{1-r}} \, .
\]
In particular, the standard 2-sphere $S^2 \subset \R^3$ is not a 
minimiser for the distortion of topologically embedded $2$-spheres in $\R^3$.
\end{proposition}

\begin{proof}
Let $w=(0,0,\sqrt{1-r^2})$ the center of the disc $D(r)$ and $v$ a point in $C(r)$ with distance $r$ from the boundary of the disc $D(r)$, see Figure~\ref{f1}. With the law of cosines we find the euclidean distance $|v-w|=\sqrt{2r-2r^2}$. This yields
\begin{align}
\label{r1}
\frac{d_{S(r)}(v,w)}{|v-w|}=\frac{2r}{\sqrt{2r-2r^2}}=\frac{\sqrt{2}}{\sqrt{1-r}} \, .
\end{align}
Let $u=(-v_1,-v_2,v_3)$ be the opposite point of $v=(v_1, v_2, v_3)$ on $S(r)$, see Figure~\ref{f1}. Using an isometric parametrization of $C(r)$ we find the intrinsic distance $d_{S(r)}(u,v)=2\sin(r \pi/2)|v|$ and get
\begin{align}
\label{r2}
\frac{d_{S(r)}(u,v)}{|u-v|}= \frac{2\sin(r \pi/2)|v|}{2r|v|}=\frac{\sin(\pi r /2)}{r} \, .
\end{align}

A geometric argument (also using Remark~\ref{twolines_rem} below) reveals that the distortion of $S(r)$ is in fact equal to the maximum of the ratios~(\refeq{r1}) and~(\refeq{r2}), that is
\begin{align*}
\delta(S(r))=\max \left({\tfrac{d_{S(r)}(u,v)}{|u-v|}, \tfrac{d_{S(r)}(v,w)}{|v-w|}}\right)=\max \left({\tfrac{\sin(\pi r /2)}{r},\tfrac{\sqrt{2}}{\sqrt{1-r}}}\right) \, .
\end{align*}
The function $r \mapsto \sin(\pi r /2)/r$ is strictly decreasing on $(0,1)$ and converges to $\pi/2$ as $r$ tends to $0$, whereas the function $r \mapsto \sqrt{2}/\sqrt{1-r}$ is strictly increasing on $(0,1)$ and converges to $\sqrt 2$ as $r$ tends to $0$. Therefore the maximum of these two functions is strictly smaller than $\pi/2$ if and only if 
\begin{align*}
\frac{\sqrt{2}}{\sqrt{1-r}}<\frac{\pi}{2} \Longleftrightarrow r< \frac{\pi^2-8}{\pi^2} \, ,
\end{align*}
and $\delta(S(r))$ achieves its unique minimum on $(0,1)$ at the solution $r_0=0.169\ldots$ of
\begin{align*}
\frac{\sin(\pi r /2)}{r}=\frac{\sqrt{2}}{\sqrt{1-r}} \, .
\end{align*}
This yields Proposition~\ref{cone}.
\end{proof}

For the second part of Theorem~1, we consider the boundary $\partial S$ of the regular $(n+1)$-simplex
$$S=\{x \in \R^{n+2}: x_1+\ldots + x_{n+2}=1, \ x_1, \ldots, x_{n+2}\geq 0\}$$ 
in $\R^{n+2}$. First, we show that $\delta(\partial S)\geq \sqrt{2+2/n}$. Consider the midpoints 
\begin{align*}
p=\tfrac{1}{n+1}(0,1,1, \ldots, 1) \text{ and } q=\tfrac{1}{n+1}(1,0,1,1,\ldots, 1)
\end{align*}
of the two $n$-facets of $S$ obtained by intersecting $S$ with the hyperplanes defined by  $x_1=0$ and $x_2=0$ respectively. For both points 
the distance to the boundary of the respective $n$-facet is intrinsically realized by the linear segment which connects $p$ and $q$ respectively to the midpoint
\begin{align*}
r=\tfrac{1}{n}(0,0,1,1, \ldots, 1)
\end{align*}
of the $(n-1)$-facet contained in the subspace defined by $x_1=x_2=0$.
In particular we get $d_{\partial S}(p,q)=|p-r|+|r-q|=2|p-r|$. A simple computation yields $|p-r|=\tfrac{1}{\sqrt{n(n+1)}}$, $|p-q|= \tfrac{\sqrt 2}{n+1}$
and hence
\begin{align*}
\label{d}
\delta(\partial S)\geq \frac{d_{\partial S}(p,q)}{|p-q|} =  \frac{2|p-r|}{|p-q|}= \sqrt{2+2/n} \, .
\end{align*}

Next, we show that $\delta(\partial S)\leq \sqrt{2+2/n}$: let $p,q \in \partial S$ be two arbitrary points. 
In the non-trivial case where $p$ and $q$ are not contained in the same $n$-facet of $S$, 
the intrinsic distance of $p$ and $q$ is bounded above by $|p-r|+|q-r|$ for any point $r$ in the common $(n-1)$-facet of the $n$-facets containing $p$ and $q$ respectively. 
It is not hard to see that for a suitable choice $r$ the angle between $p-r$ and $q-r$ is at least $\alpha=\cos^{-1}(\tfrac{1}{n+1})$, the angle between two $n$-facets of $S$. 
By Remark~\ref{twolines_rem} below, the ratio of intrinsic to extrinsic distance of $p$ and $q$ is bounded above by $2|v|/|v-w|$ for vectors $v$ and $w$ of equal length and angle $\alpha=\cos^{-1}(\tfrac{1}{n+1})$. 
A computation shows that this upper bound for the ratios of intrinsic to extrinsic distance is precisely $\sqrt{2+2/n}$, as desired. 

\begin{remark}
	\label{twolines_rem}
The distortion of a union of two rays $R$ with a common point $r$ is realized by pairs of points $p$ and $q$ at the same distance from $r$ and is equal to $\delta(R) = \frac{|p-r| + |q-r|}{|p-q|}$.
\end{remark}

\section{Convex minimisers}

In this section we prove Theorem~2 stated in the introduction. It is crucial in the proof that the distortion of a convex 2-sphere is strictly smaller than $\pi/2$. The proof below shows the existence of convex minimisers and uniform boundedness of their eccentricity. As mentioned in the introduction, this contrasts the situation for the class of 2-spheres with distortion $\leq\pi/2$, which can have arbitrarily large eccentricity.

\begin{proof}[Proof of Theorem~2]
We employ the direct method of the calculus of variations to find a minimiser in this class. Let $\mathcal C$ be the class of compact and convex subsets of $\R^3$ and $\mathcal C' \subset \mathcal C$ be the subclass of those sets with nonempty interior. Let $(K_n)$ be a sequence in $\mathcal C'$ such that
$$S \mathrel{\mathop:}= \lim_{n \to \infty} \delta(\partial K_n) = \inf\{\delta(\partial K) : K \in \mathcal C'\} \, .$$
Since the distortion is translation and scaling invariant we may assume that $\operatorname{diam}(K_n) = 2$ and the diameter is achieved at the points $\pm p \mathrel{\mathop:}= (\pm 1, 0,0) \in K_n$. With this normalization, the Blaschke selection principle guarantees a subsequence of $(K_n)$ that converges in Hausdorff distance to some $K \in \mathcal C$, see \cite{B} or \cite[Theorem~7.3.8, Remark~7.3.9]{BBI} for a modern reference. Without loss of generality we assume that $(K_n)$ already converges to $K$. In two steps we show that $K$ has nonempty interior and then that $S = \delta(\partial K)$.

First it is easy to see that $\operatorname{diam}(K) = 2$ and $\pm p \in K$. Assume by contradiction that $K$ has empty interior. Then $K$ is contained in a two-dimensional plane $V$. This plane has to contain the $x$-axis, rotating all the sets $(K_n)$ we may assume that $K$ is contained in the $(x,y)$-plane $V$. Now either $K$ has nonempty interior in $V$ or $K = [-p,p]$.

In the first case we find a point $c = (c_x,c_y,0) \in V \cap K$ and an open disc $D \subset V \cap K$ with center $c$ and some radius $r > 0$. Because of the convergence of $K_n$ to $K$, the orthogonal projections $\pi_V(K_n)$ also converge to $\pi_V(K) = K$ and if $n$ is big enough, say $n \geq N$, then $\pi_V(K_n)$ also contains the disc $D$. Moreover there are $a_n,b_n > 0$ that converge to $0$ for $n \to \infty$ such that $q_n^{+} \mathrel{\mathop:}=(c_x,c_y,a_n),q_n^{-} \mathrel{\mathop:}=(c_x,c_y,-b_n) \in \partial K_n$. For any $n \geq N$, the preimage $\pi_V^{-1}(D) \cap \partial K_n$ is composed of two disjoint sets and $\pi_V$ is a homeomophism onto $D$ on each of them. Thus if $\gamma : [0,1] \to \partial K_n$ is a curve in $K_n$ that connects $q_n^{+}$ and $q_n^{-}$ in $\partial K_n$, then $\pi_V \circ \gamma$ is a curve in $V$ of length at least $2r$. Because $\pi_V$ is $1$-Lipschitz, $\gamma$ has length at least $2r$ and it follows
\[
\delta(\partial K_n) \geq \frac{d_{\partial K_n}(q_n^{+},q_n^{-})}{|q_n^{+}-q_n^{-}|} \geq \frac{2r}{a_n + b_n} \to \infty
\]
for $n \to \infty$. This contradicts the choice of the sequence $(K_n)$.

In the second case $K = [-p,p]$. Let $W$ be the $(y,z)$-plane and consider the intersection $K_n' \mathrel{\mathop:}= K_n \cap W$. Each $K_n'$ is a compact convex set with nonempty interior in $W$ and a boundary circle $C_n$. As an embedded circle $C_n$ has distortion at least $\pi/2$, see \cite{G1}. So there are two different points $s_n,t_n \in C_n$ with
\[
\frac{L([s_n,t_n])}{|s_n-t_n|} \geq \frac{\pi}{2} \, ,
\]
where $L([s_n,t_n])$ is the smaller length of the two arcs in $C_n$ that connects $s_n$ and $t_n$. Now since $K_n$ is convex it contains the cone $A^\pm_n$ with base $K_n'$ and vertex $\pm p$. Set $A_n \mathrel{\mathop:}= A^+ \cup A^-$. This is the union of two cones with the same base, so $A_n$ is itself compact, convex and contains the points $\pm p$. Let $\pi_n : \partial K_n \to A_n$ be the nearest point projection. Because $A_n$ is contained in $K_n$, $\pi_n$ is a $1$-Lipschitz map from $\partial K_n$ onto $\partial A_n$. Because $s_n,t_n \in \partial A_n$ it follows that
$$d_{\partial K_n}(s_n,t_n)\geq d_{\partial A_n}(s_n,t_n) \, .$$
By assumption, $r_n \mathrel{\mathop:}= \sup_{t \in C_n} |t|$ converges to $0$ for $n \to \infty$. Because of Proposition~\ref{cone} we know that $S < \pi/2$, so we may assume that $\delta(\partial K_n) < \pi/2-\epsilon$ for all $n$ and some $\epsilon > 0$. Thus for any $n$ there is a curve $\gamma_n$ connecting $s_n$ and $t_n$ in $\partial A_n$ such that
$$L(\gamma_n) < (\pi/2-\epsilon)|s_n-t_n| \leq (\pi/2-\epsilon)2r_n \leq \pi r_n \, .$$
Since $A_n$ is mirror symmetric with respect to $W$ we may assume that the image of $\gamma_n$ is contained in $A_n^{+}$. Because of the estimate above the projection of the curve $\gamma_n$ onto the $x$-axis is contained in the interval $[0,\pi r_n]$. Let $P : \partial A_n \cap ([0,\pi r_n] \times \R^2) \to W$ be the radial projection from the point $p$ onto $W$. If $n$ is big enough, then $L(P\circ \gamma_n) \leq (1 + 6\pi r_n)L(\gamma_n)$ (see Lemma~\ref{lipschitz_lem} below). Since $P\circ \gamma_n$ connects $s_n$ and $t_n$ in $C_n$ it follows that
\begin{align*}
L([s_n,t_n]) & \leq L(P\circ \gamma_n) \leq (1 + 6\pi r_n)L(\gamma_n) \\
 & \leq (1 + 6\pi r_n)(\pi/2-\epsilon)|s_n-t_n| \, .
\end{align*}
But this contradicts $L([s_n,t_n]) \geq \pi/2 |s_n-t_n|$ if $n$ is big enough.

With this preparation we know that $K$ has nonempty interior and therefore that $K \in \mathcal C'$. It remains to show that $\delta(K) = S$. First it is clear that $\delta(K) \geq S$ by the definition of $S$. The other inequality follows directly from the lower semicontinuity of $\delta$. Here is a proof for the convenience of the reader:
Fix two points $x,y \in \partial K$. By construction there are sequences $x_n,y_n \in \partial K_n$ that converge to $x$ and $y$ respectively, because $\partial K_n$ converges in Hausdorff distance to $\partial K$ (which is implied by the convergence of $K_n$ to $K$ because $K$ has nonempty interior). Choose curves $\gamma_n : [0,1] \to \partial K_n$ with $L(\gamma_n) \leq (\delta(K_n) + \frac{1}{n})|x_n-y_n|$. We can assume that each $\gamma_n$ is parametrized proportional to arc length, see \cite[Proposition~2.5.9]{BBI}. With this parametrization, the sequence $\gamma_n$ has a uniform bound on the Lipschitz constants (namely $\sup_{n \geq 1} (\delta(K_n) + \frac{1}{n})|x_n-y_n|$) and due to the theorem of Arzel\`a-Ascoli there exists a subsequence $(\gamma_{n_k})$ that converges uniformly to some Lipschitz curve $\gamma : [0,1] \to \R^3$ that connects $x$ with $y$. The length of $\gamma$ is estimated by
\begin{align*}
L(\gamma) & \leq \operatorname{Lip}(\gamma) \leq \limsup_{n\to\infty}\operatorname{Lip}(\gamma_n) \\
& \leq \limsup_{n\to\infty}(\delta(K_n)+\tfrac{1}{n})|x_n-y_n| = S|x-y| \, .
\end{align*}
Because $\partial K_n$ converges to $\partial K$ it follows that the image of $\gamma$ is contained in $\partial K$. The points $x,y \in \partial K$ are arbitrary and therefore we conclude $\delta(K) \leq S$.
\end{proof}

Here is a technical lemma used in the proof above:

\begin{lemma}
	\label{lipschitz_lem}
$L(P\circ \gamma_n) \leq (1 + 6\pi r_n)L(\gamma_n)$ if $6\pi r_n \leq 1$.
\end{lemma}

\begin{proof}
Assume that $n$ is big enough such that $6\pi r_n \leq 1$. The map $P$ is given by
\[
P(x,y,z) = \frac{1}{1-x}(0,y,z) \, .
\]
We estimate the Lipschitz constant of $P$ on the cylinder $Z = [0,\pi r_n]\times B^2(0,r_n)$. For $q = (x,y,z), q' = (x',y',z') \in Z$ it holds
\begin{align*}
& |P(q) - P(q')| = \biggl|\frac{(y,z)}{1-x} - \frac{(y',z')}{1-x'}\biggr| = \biggl|\frac{(y,z)(1-x') - (y',z')(1-x)}{(1-x)(1-x')}\biggr| \\
 & \quad\leq \frac{1}{(1-\pi r_n)^2} |(y,z)(1-x') - (y',z')(1-x)| \\
 & \quad\leq \frac{1}{(1-\pi r_n)^2} \bigl(|(y,z) - (y',z')| + |(y,z)x' - (y,z)x| + |(y,z)x - (y',z')x|\bigl) \\
 & \quad\leq \frac{1}{(1-\pi r_n)^2} \bigl(|(y,z) - (y',z')| + r_n|x' - x| + \pi r_n|(y,z) - (y',z')|\bigl) \\
 & \quad\leq \frac{1 + 2\pi r_n}{(1-\pi r_n)^2}|q-q'| \leq \frac{1 + 2\pi r_n}{1-2\pi r_n}|q-q'| \leq (1 + 6\pi r_n)|q-q'| \,.
\end{align*}
In the last estimate we used that $\frac{1 + s}{1-s} \leq 1 + 3s$ if $s \in [0,\frac{1}{3}]$. The estimate for the length of the curves follows immediately.
\end{proof}

\section{Lower bounds}

In this section we derive lower bounds on the distortion of closed subsets of $\R^n$. The third one requires the existence of a systole of a subset $A \subset \R^n$, i.e. a non-contractible loop of minimal length in $A$. All these bounds have been discovered and proved by Gromov and Pansu, see Chapter~1 and Appendix~A in~\cite{G2}. We provide proofs since they are short and instructive.

\begin{proposition}
\label{lower bounds}
Let $A$ be a closed subset of $\R^n$.
\begin{enumerate}
\item[(i)] If the complement of $A$ has a bounded component, then the distortion of $A$ is at least $\frac{\pi}{2\sqrt{2}}$. 
\item[(ii)] If $A=-A$ and $A$ does not contain the origin, then the distortion of $A$ is at least $\pi/2$. 
\item[(iii)] If $A$ has a systole, then the distortion of $A$ is at least $\pi/2$.
\end{enumerate}
\end{proposition}

In all three statements, closedness of the subset $A$ is essential, as shows the example $A=\R^n \setminus \{0\} \subset \R^n$, whose distortion is one. The following strengthening of the first statement is derived in \cite[Section~1.14]{G2}: Let $A \subset \R^n$ be compact with distortion $\delta(A) < \tau_n$, for some fixed $\tau_n \in (\frac{\pi}{2\sqrt{2}},\frac{\pi}{2})$, then $A$ is contractible.

The third inequality applies to all compact submanifolds of $\R^n$ with non-trivial fundamental group, since these admit systoles.  

\begin{proof}
For (i), let $C$ be a bounded component of $\R^n \setminus A$ and let $B$ be an open ball of maximal radius contained in $C$. 
Since the distortion is invariant under scalings and translations we may assume that $B$ is the open unit ball centered at the origin. 
Let $p \in \overline{B} \cap A$ and assume for a contradiction that for every $q \in \overline B \cap A$ the angle between $p$ and $q$ is smaller than $\pi/2$. 
This implies that $\overline B \cap A$ is contained in the hemisphere $\{b \in \partial B: \langle p, b \rangle > 0\}$. 
For sufficiently small $t >0$ the closed ball $\overline B-tp$ is also contained in the open component $C$, 
a contradiction to the maximality of the radius of $B$. This shows that there exists a point $q \in \overline{B} \cap A$ such that the angle between $p$ and $q$ is at least $\pi/2$. 
We conclude 
\begin{align*}
\delta(A)\geq \frac{d_A(p,q)}{|p-q|} \geq \frac{d_{\partial B}(p,q)}{|p-q|} \geq \frac{\pi}{2\sqrt 2} \, ,
\end{align*}
where we used that the standard projection $\R^n \setminus B \to \partial B$ is $1$-Lipschitz and that 
the angle between $p$ and $q$ is at least $\pi/2$. This finishes the proof of (i).

For (ii), let $\overline B$ be a closed ball centered at the origin with minimal positive radius $r$ such that $\overline B \cap A$ is non-empty. By the symmetry of $A$ there are 
antipodal points $\pm p$ contained in $\overline B \cap A$. We conclude 
\begin{align*}
\delta(A)\geq \frac{d_A(p,-p)}{|p-(-p)|} \geq \frac{d_{\partial B}(p,-p)}{2|p|} = \frac{\pi r}{2r}=\frac{\pi}{2} \, ,
\end{align*}
where we used again that the standard projection $\R^n \setminus B \to \partial B$ is $1$-Lipschitz and that $\pm p$ are antipodal points on $\partial B$. This proves (ii). 

For (iii) we first show the following crucial statement: if $S\subset A$ is a systole of $A$, then $d_S(p,q)=d_A(p,q)$ for all pairs of points $p,q \in S$. 
To see this, let $\gamma \subset A$ be a path connecting two given points $p,q \in S$. Moreover, let $\gamma_1, \gamma_2 \subset S$ be the two paths connecting 
$p$ and $q$ in $S$. We have $S=\gamma_1 \cup \gamma_2$. If both closed loops $\gamma \cup \gamma_1$ and $\gamma \cup \gamma_2$ are 
contractible, then both pairs $\gamma, \gamma_1$ and $\gamma, \gamma_2$ are homotopic relative to the endpoints $p$ and $q$. In particular 
$\gamma_1, \gamma_2$ are homotopic relative to the endpoints, which contradicts the fact that $\gamma_1\cup \gamma_2=S$ is not null-homotopic. Therefore 
we can assume that $\gamma \cup \gamma_1$ is not null-homotopic and get
\begin{align*}
l(\gamma)+l(\gamma_1)=l(\gamma \cup \gamma_1)\geq l(S)=l(\gamma_1 \cup \gamma_2)=l(\gamma_1) +l(\gamma_2)
\end{align*}
by the minimality of the length of $S$ among non-null-homotopic curves. This implies $l(\gamma)\geq l(\gamma_2)\geq d_S(p,q)$. Since $\gamma \subset A$ was arbitrary we conclude 
$d_A(p,q)\geq d_S(p,q)$, as desired. With the statement $d_A(p,q)=d_S(p,q)$ for $p,q \in S$ at hand we see immediately that the distortion of $A$ is bounded below 
by the distortion of $S$. Since $S$ is a closed loop, its distortion is at least $\pi/2$. This proves (iii).
\end{proof}

\section{Surfaces with low distortion}

Thanks to Proposition~3, the distortion of an embedded closed surface of genus $g \geq 1$ in $\R^n$ is at least $\pi/2$. Indeed, closed non-simply connected surfaces admit a systole with respect to any Riemannian metric. In order to prove Theorem~2, we need to construct embeddings of surfaces with distortion arbitrarily close to $\pi/2$. We start with the closed torus, which we embed as the boundary of an $\epsilon$-neighbourhood $N_\epsilon(S^1)$ of a unit circle $S^1$ in $\R^3$. We claim that

$$\lim_{\epsilon \to 0} \delta(\partial N_\epsilon(S^1))=\pi/2 \, .$$

This follows from the fact that $\partial N_\epsilon(S^1)$ looks locally like a straight cylinder of radius $\epsilon$ and globally like a unit circle (both of which have distortion $\pi/2$). The deviation of the distortion from $\pi/2$ is governed by a linear expression in $\epsilon$. In fact, we suspect that the distortion of every individual torus $\partial N_\epsilon(S^1)$ is $\pi/2$, for all $\epsilon<1$.

For surfaces of higher genus $g$, we consider a disjoint union of $g$ tori as above, arranged at large distance along a line. We connect these tori by $g-1$ thin cylinders (say of radius $\epsilon^2$) along that  line. The large distance between consecutive tori makes sure that the intrinsic distance of pairs of points contained in two consecutive tori is close to their Euclidean distance. As for a single torus, the distortion of the resulting surface of genus $g$ tends to $\pi/2$, as $\epsilon$ tends to zero. However, the distortion of such a surface might be larger than $\pi/2$, since connecting consecutive tori by thin cylinders involves cutting out small discs from these tori and this slightly changes the inner metric. 

\bigskip
\noindent
Mathematisches Institut, Sidlerstr.~5, CH-3012 Bern, Switzerland

\bigskip
\noindent
\texttt{sebastian.baader@math.unibe.ch}

\smallskip
\noindent
\texttt{luca.studer@math.unibe.ch}

\smallskip
\noindent
\texttt{roger.zuest@math.unibe.ch}

\end{document}